\documentclass[a4paper,oneside,11pt]{article}
\usepackage{amsmath}
\usepackage{amssymb}
\usepackage{indentfirst}

\usepackage{abstract}
\usepackage{titlesec}
\usepackage{graphicx}
\usepackage{fancyhdr}
\usepackage{latexsym}
\usepackage{mathrsfs}
\usepackage{booktabs}
\usepackage{mathrsfs}
\usepackage{mathtools}
\usepackage{amsthm}
\usepackage{wasysym}
\usepackage{cite}
\usepackage{color}
\usepackage{todonotes}
\usepackage{geometry}

\newtheorem{theorem}{Theorem}[section]
\newtheorem{remark}{Remark}[section]

\newtheorem{definition}{Definition}[section]
\newtheorem{lemma}[theorem]{Lemma}

\newtheorem{proposition}[theorem]{Proposition}

\newcommand{\n}{\rho}

\newcommand{\mr}{\mathbb{R}}
\newcommand{\lm}{\lambda}

\renewcommand{\div}{ {\rm div }  }

\newcommand{\pa}{\partial}
\renewcommand{\r}{\mathbb{R}}

\renewcommand{\b}{B_{N_*}}

\newcommand{\ia}{\int_0^T}

\newcommand{\bt}{\begin{theorem}}
\newcommand{\bl}{\begin{lemma}}
\newcommand{\el}{\end{lemma}}
\newcommand{\et}{\end{theorem}}
\newcommand{\ga}{\gamma}

\newcommand{\curl}{{\rm curl} }

\newcommand{\al}{\alpha}
\newcommand{\de}{\delta}
\newcommand{\ve}{\varepsilon}
\newcommand{\la}{\label}

\newcommand{\ol}{\overline}

\newcommand{\bn}{\begin{eqnarray}}
\newcommand{\en}{\end{eqnarray}}
\newcommand{\bnn}{\begin{eqnarray*}}
\newcommand{\enn}{\end{eqnarray*}}

\newcommand{\bnnn}{\begin{eqnarray*}}
\newcommand{\ennn}{\end{eqnarray*}}

\newcommand{\ba}{\begin{aligned}}
\newcommand{\ea}{\end{aligned}}
\newcommand{\be}{\begin{equation}}
\newcommand{\ee}{\end{equation}}

\def\O{{\r^2 }}
\def\p{\partial}
\def\norm[#1]#2{\|#2\|_{#1}}

\def\o{\omega}

\newcommand{\no}{\nonumber\\}

\newcommand{\si}{\sigma}

\def\la{\label}

\def\na{\nabla}
\def\on{\bar\n}

\makeatletter      
\@addtoreset{equation}{section}
\makeatother

\makeatletter      
\@addtoreset{equation}{section}
\makeatother       

\title{
Global Well-Posedness and Large Time Asymptotic Behavior of Strong Solutions to the Compressible Viscous Non-Resistive MHD Equations with Vacuum}

\author{Jiaxu L{\small I}$^{a}$, Boqiang L{\small \"U}$^{b}$, Bing Y{\small UAN}$^{b}$, \thanks{email: jiaxvlee@gmail.com (J.X. Li), lvbq86@ncu.edu.cn \& lvbq86@163.com (B.Q.   L\"u), bingyuan1201@163.com (B. Yuan)} \\
{\normalsize a. School of Mathematical Sciences,}\\
{\normalsize Shenzhen University, Shenzhen 518061, P. R. China;}\\
{\normalsize b. School of Mathematics and Computer Sciences,}\\ {\normalsize \& Institute of Mathematics and Interdisciplinary Sciences,}\\ {\normalsize  Nanchang University, Nanchang 330031, P. R. China;}}
\date{ }

\begin{document}
\maketitle
\begin{abstract}
In this paper, we consider the  Cauchy problem to 
the compressible viscous non-resistive magnetohydrodynamic (MHD) system in $\r^3$, in which the velocity is planar while the magnetic field is vertical. 
For regular initial data with small energy but possibly large oscillations, we establish the global existence and uniqueness of strong solution with vacuum. In addition, the large-time behavior of the solution is investigated. 
A key ingredient is an application of the Bourgain--Brezis divergence estimate to handle the total pressure $\rho^\gamma+m^2$ for any adiabatic index $\gamma>1$. This enables us to control both components of the total pressure, $\rho^\gamma$ and $m^2$, by a single estimate that remains valid in the presence of vacuum. 
To the best of our knowledge, our result is the first one on global solvability to the compressible, viscous,
non-resistive MHD system in the special 2$\frac12$-dimensional whole space allowing vacuum.
\end{abstract}

\section{Introduction}
The motion of a three-dimensional (3D) viscous compressible fluid under the effect of electromagnetic field is governed by the compressible MHD equations:
\begin{equation}\label{1}
\left\{ \begin{array}{l}
  \rho_t+\mathrm{div}(\rho v)=0,\\
  (\rho v)_t+\mathrm{div}(\rho v\otimes v)-\mu\Delta v-(\mu+\lambda)\nabla\div v+\nabla P=(\curl H)\times H, \\
  H_{t}- \eta \Delta H=\mbox{curl}(v\times H),\quad\\
  \mbox{div}H=0.
  \end{array} \right.
\end{equation}
Here $t\geq0$ is time, $x=(x_1,x_2,x_3)\in\Omega\subset\mathbb{R}^3$ is the spatial coordinate,  
the unknown functions $\rho$, $v$, $P(\rho)=A\rho^\gamma (A>0, \gamma>1)$, and $H$ denote the density, velocity, pressure, and the magnetic field, respectively.  The constant
viscosity coefficients $\mu$ and $\lambda$ satisfy the physical
restrictions: 
\be\no
\mu>0,\quad 2\mu +3\lambda\ge 0.
\ee
The constant $\eta\geq 0$ is the magnetic diffusivity describing a magnetic diffusion coefficient of the magnetic field. Without loss of generality, we assume that $A=1.$

This paper focuses on a special 2$\frac12$-D compressible viscous non-resistive MHD system. The motion of fluids takes place in the plane while the magnetic field acts on fluids only in the vertical direction, namely
\be\no
v=(v^1(t,x_1,x_2),v^2(t,x_1,x_2),0)\triangleq(u,0),\quad \rho\triangleq\rho(t,x_1,x_2),\quad H\triangleq(0,0,\sqrt2m(t,x_1,x_2)).
\ee
Then \eqref{1} reduces to
\be \la{a1}
\begin{cases}
\n_t + \div(\n u) = 0,\\
(\n u)_t + \div(\n u\otimes u) -\mu\Delta u - (\mu + \lambda)\nabla\div u+ \nabla P+\nabla m^2 = 0,\\
m_t+\div(mu)=0.
\end{cases}
\ee

We consider  the Cauchy problem
for (\ref{a1}) with $(\n,u,m)$  vanishing at infinity (in some weak sense) with given initial data $\n_0$, $u_0,$  and $m_0$ given by
\be
\la{h2} \n(x,0)=\n_0(x), \quad \n u(x,0)=\n_0u_0(x),\quad m(x,0)=m_0, \quad x\in  \r^2.
\ee

The MHD system \eqref{1} models the dynamics of electrically conducting fluids, such as plasmas, and incorporates the effects of magnetic fields on fluid motion. If this motion occurs in the absence of a magnetic field, that is $H =0$, the MHD system \eqref{1} becomes the well-known Navier-Stokes equations for compressible fluids, which have been studied extensively in  \cite{Nash1962,Serrin1959,Tani19774,3,Straskraba1993,matsumura1980initial,Lions1998-2,Feireisl2001,jiang1,jiang2,HuangLiXin2012,Li201977} and the references therein. For the constant viscous flows, the local well-posedness of smooth solutions is established in \cite{Nash1962,Serrin1959,Tani19774} (in the absence of vacuum)  and \cite{3,Straskraba1993} (with initial vacuum).
The global classical solutions were first obtained by Matsumura-Nishida \cite{matsumura1980initial} for initial data close to a non-vacuum equilibrium in $H^3$. In the presence of initial vacuum, the global existence of solutions becomes much more challenging due to the degeneration of momentum equation. The major breakthrough is due to Lions \cite{Lions1998-2}, who obtained the global existence of finite-energy weak solutions when the adiabatic exponent $\gamma\ge9/5$ for the 3D case, which was further refined to the critical case $\gamma > 3/2$ by Feireisl \cite{Feireisl2001}. For spherically symmetric or axisymmetric initial data, Jiang-Zhang \cite{jiang1,jiang2} proved the global existence of weak solutions with vacuum for any $\gamma>1$. 
For the Cauchy problem,
Huang-Li-Xin \cite{HuangLiXin2012} and Li-Xin \cite{Li201977}  established the  3D and 2D global  classical solution  with  small total energy but possibly large oscillations. 
More related issues concerning the global   solutions with smallness initial data can be found in \cite{zcj1,why1,ljk1} and references therein.

Then, back to the MHD system \eqref{1}, whose mathematical structure and physical mechanism are more complicated due to the fluid dynamic motion and the magnetic field interact strongly on each other. Here, we briefly review some results on the well-posedness of solutions with vacuum,
which are more closely related to the topic of this paper. When the resistive coefficient $\eta>0$, the local existence of strong solutions with vacuum is shown in \cite{Fan2009,Lv2015}. 
Hu-Wang \cite{huwang2010} obtained the global existence of finite energy renormalized weak solutions with suitably large adiabatic exponent $\gamma$ in 3D bounded domains. Liu-Yu-Zhang \cite{lyz2013jde} established the global weak solutions with small total energy to the 3D Cauchy problem when the far-field density is away from vacuum. For the Cauchy problem with far-field vacuum, Li-Xu-Zhang \cite{lxz2013siam} and L\"u-Shi-Xu \cite{lsxiumj}  proved the global 3D and 2D classical solutions with initial data which are of small energy but possibly large oscillations respectively. 

When the resistive coefficients $\eta=0$, the mathematical analysis of \eqref{1} becomes more challenging owing to the lack of a dissipation mechanism for the magnetic field. Jiang-Zhang \cite{10} proved the global well-posedness of strong solutions for large data and studied the non-resistive limit in 1D case.  For the 2D system, Li-Sun \cite{19} proved the global existence of weak solutions for non-isentropic cases, while Wu-Wu \cite{27} and Wu-Zhu \cite{29} obtained global strong solutions for small perturbations on $\mathbb{R}^{2}$ and $\mathbb{T}^{2}$,
respectively. Zhu \cite{34} established local
classical solutions to the 3D Cauchy problem allowing vacuum. Tan-Wang \cite{24} and Li \cite{16} established the global existence of smooth solutions in a horizontally infinite flat layer near a strong vertical magnetic field.  
Wu-Zhai \cite{28} and Li- Xu-Zhai \cite{20} obtained global strong 
solutions  on $\mathbb{T}^3$ under assumption that the initial data are sufficiently close to an equilibrium state.


Recently, the global well-posedness of solutions to a special 2$\frac12$-dimensional system \eqref{a1} has been considered in \cite{dong,dong1,zhai2025}.  Dong-Wu-Zhai \cite{dong} proved the global existence of strong solutions with small initial data in the critical Besov spaces, see also \cite{dong1} subsequently proved stability
and exponential decay
  of smooth solutions near a steady-state in $\mathbb{T}^2$. Zhai-Zhang \cite{zhai2025} 
 extended the
near-equilibrium analysis \cite{dong}  to the full non-isentropic system.  
All these global results \cite{dong,dong1,zhai2025} hold for strictly positive density. Therefore, it is natural to ask
whether the Cauchy problem of \eqref{a1}  on $\mathbb{R}^{2}$ admits global
strong solutions when the initial density may vanish.   In fact, this is the main aim  of our  paper.

Before stating the main results, we first explain the notations and conventions used throughout this paper. For $R>0$, set
$$
B_R  \triangleq\left.\left\{x\in\r^2\right|
\,|x|<R \right\} , \quad \int fdx\triangleq\int_{\r^2}fdx.
$$
Moreover, for $1\le r\le \infty, k\ge 1,$ the standard homogeneous and inhomogeneous Sobolev spaces are defined as follows:
\bnn
\begin{cases}
L^r=L^r(\r^2 ),\quad D^{k,r}=D^{k,r}(\r^2)=\{v\in L^1_{\rm loc}(\r^2)| \na^k v\in L^r(\r^2)\}, \\
D^1=D^{1,2},\quad W^{k,r}  = W^{k,r}(\r^2) , \quad H^k = W^{k,2} , \\
\dot H^\beta=\left\{f:\r^2\rightarrow \r\left|\|f\|^2_{\dot H^\beta}=\displaystyle{\int} |\xi|^{2\beta}|\hat f(\xi)|^2d\xi<\infty\right.
\right\},
\end{cases}
\enn
where $\hat f$ is the Fourier transform of $f.$ Next, we give the definition of strong solution to \eqref{a1} as follows:
\begin{definition}
If all derivatives involved in \eqref{a1} for $(\rho,u,m)$ are regular distributions satisfying \eqref{1.10}, and equations \eqref{a1} hold almost everywhere in $\r^2\times (0,T),$ then $(\n,u,m)$  is called a strong solution to \eqref{a1}.
\end{definition}

The initial total energy is defined as:
\bnn
C_0 =\int_{\r^2}\left(\frac{1}{2}\n_0|u_0|^2 + \frac{1}{\ga-1}P(\n_0)+m_0^2 \right)dx.
\enn
Without loss of generality, assume that the initial density $\n_0$ satisfies
\be\la{oy3.7}
\int_{\r^2} \n_0dx=1,
\ee
which implies that there exists a positive constant $N_0$ such that
\be\la{oy3.8}
\int_{B_{N_0}}  \n_0  dx\ge \frac12\int\n_0dx=\frac12.
\ee

We can now state our main result in this paper, concerning the global existence  of  strong solutions.
\begin{theorem}\la{th1} In addition to \eqref{oy3.7} and  \eqref{oy3.8},  suppose that  the initial data $(\n_0,u_0,m_0)$ satisfy  for any given numbers $M>0,$ $\on\geq 1,\bar m\ge 1,$   $a>1,~q>2$, and $\beta\in(0,1]$,
\begin{equation} \la{co1}
\left\{\begin{array}{lll}
0\le \n_0\le\bar{\n},\quad |m_0|\leq \bar m, \quad  \bar x^a \rho_0\in   L^1 \cap H^1\cap W^{1,q},\\
u_0\in {\dot H}^\beta\cap D^1,\quad \bar x^{a}m_0\in L^1,\quad m_0\in H^1\cap W^{1,q},\quad \n_0^{1/2}u_0  \in L^2,
\end{array}\right.
\end{equation}
and that
\be\la{h7}
\|u_0\|_{{\dot H}^\beta}+ \|\n_0\bar x^a\|_{L^1}+\|m_0\bar x^{a}\|_{L^1}\le M ,
\ee
where
\be\la{2.07}
\bar x\triangleq(e+|x|^2)^{1/2} \log^2 (e+|x|^2).
\ee
Then there exists a positive constant $\ve$ depending on $\mu, \lambda, \ga, a, \on, \bar m, \beta, N_0,$ and $M$ such that if
\be\la{i7}
C_0\le\ve,
\ee
the problem  \eqref{a1}--\eqref{h2} has a unique global strong solution $(\n,u,m)$ satisfying   for any $0<T<\infty,$
\be\la{h8}
0\le\n(x,t)\le 2\bar{\n},\quad |m(x,t)|\le 2\bar m,\quad  (x,t)\in \O\times[0,T],\quad\quad\quad\quad\quad
\ee
\be\la{1.10}
\begin{cases}
\rho\in C([0,T];L^1 \cap H^1\cap W^{1,q} ),\quad m\in C([0,T];H^1\cap W^{1,q}),\\
\bar x^a\rho\in L^\infty( 0,T ;L^1\cap H^1\cap W^{1,q} ),\quad \bar x^a m\in L^\infty( 0,T ;L^1),\\
\sqrt{\n } u,\,\na u,\, \bar x^{-1}u,\,    \sqrt{t} \sqrt{\n}  u_t \in L^\infty(0,T; L^2 ), \\
\na u\in  L^2(0,T;H^1)\cap  L^{(q+1)/q}(0,T; W^{1,q}), \\
\sqrt{t}\na u\in L^2(0,T; W^{1,q} ),  \\
\sqrt{\n} u_t,  \,  \sqrt{t}\na u_t ,\, \sqrt{t} \bar x^{-1}u_t\in L^2(\O\times(0,T)),
   \end{cases}\ee
and
\be\la{l1.2}
\inf\limits_{0\le t\le T}\int_{B_{N_1(1+t)\log^\al(e+t) }}\n(x,t) dx\ge \frac14,
\ee
for any $\al>1$  and  some positive constant $N_1$ depending only on $\al,N_0, $ and $M.$ Moreover, $(\n,u, m)$ has the following decay
rates,  that is,  for $t\ge 1,$
 \be \la{sj}
\begin{cases}
\|\na u(\cdot,t)\|_{L^p}  \le C(p)t^{-1+1/p},\quad  for~p\in[2,\infty),\\
\|P(\cdot,t)\|_{L^r}+\|m^2(\cdot,t)\|_{L^r}\le C(r)t^{-1+1/r},\quad  for~r\in[2,\infty),\\
\|\rho(\cdot,t)\|_{L^r}\le C(r)t^{-\frac{r-1}{r(2\gamma-1)}},\quad for~r\in(1,2\gamma],\\
\|m(\cdot,t)\|_{L^r}\le C(r)t^{-\frac{r-1}{3r}},\quad  for~r\in(1,4],\\
\|\na \omega(\cdot,t)\|_{L^2}+\|\na F(\cdot,t)\|_{L^2}\le Ct^{-1},
\end{cases} \ee
where
\be \la{hj1}
\o\triangleq \pa_1u^2-\pa_2u^1 ,\quad F\triangleq(2\mu+\lambda)\div u-P-m^2,
\ee
are respectively the vorticity and the so-called effective viscous flux, and $C$ depends on $\mu, \lambda, \ga, a,$ $\on,$ $\bar m,$ $\beta,$ $N_0,$ and $M.$
\end{theorem}

\begin{remark}
It is worth noting that we consider global solutions to the Cauchy problem for system \eqref{a1} in the presence of vacuum. Unlike the situation in \cite{dong1}, the equations for the velocity $u$ and the perturbation of $m$ do not possess wave structures. Consequently, the method adopted in this paper is fundamentally different from that in \cite{dong1}.
\end{remark}


\begin{remark}
    When $\gamma\geq 2$, the initial condition ${\bar x}^a\rho_0\in D^1\cap D^{1,q}$ can  be removed. Indeed, since $\nabla P=\gamma\rho^{\gamma-1}\nabla\rho$, the $\rho^{\gamma-1}$ $(\gamma-1\geq1)$ is sufficient to absorb the positive power weights generated by $u$. 
    However, it is worth noting that the regularity result $\bar x^a\rho\in L^\infty( 0,T ;D^1\cap W^{1,q} )$ will not be obtained.
\end{remark}

{\begin{remark}
   Although the magnetic field has no diffusive term, Theorem~1.1 still yields the algebraic decay (see \eqref{sj}). Specifically, the decay estimates for $P$ and $m^2$ can be derived from the following inequalities: for $t\geq 1,$
\begin{align}
&\|\rho\|_{L^r}\leq C\|\rho\|_{L^1}^\frac{2\gamma-r}{r(2\gamma-1)}\|\rho^\gamma\|_{L^2}^\frac{2(r-1)}{r(2\gamma-1)}\leq Ct^{-\frac{r-1}{r(2\gamma-1)}},\quad\text{for}\quad 1<r\leq2\gamma,\\
&\|m\|_{L^r}\leq C\|m\|_{L^1}^\frac{4-r}{3r}\|m^2\|_{L^2}^\frac{2(r-1)}{3r}\leq Ct^{-\frac{r-1}{3r}},\quad\text{for}\quad 1<r\leq4,
\end{align}
where we used $\int|m|dx=\int|m_0|dx$, which is directly derived from \eqref{mfc}. It is worth noting that the decay rate of $\|P(\cdot,t)\|_{L^r}$ (i.e. \eqref{sj}$_3$) is sharper than that in \cite[(1.15)]{Li201977}, while the decay rate of the magnetic field differs from that for the compressible viscous resistive MHD equations \cite{lsxiumj}.
\end{remark}}

{\begin{remark}
When $m=0$ (i.e., there is no magnetic field), the system \eqref{a1} reduces to the 2D compressible Navier-Stokes equations, and  Theorem \ref{th1}  reduces to the corresponding results of Li-Xin \cite{Li201977}. Roughly speaking, we generalize  the results of \cite{Li201977} to the special 2$\frac12$-D compressible viscous non-resistive MHD equations.
\end{remark}}



We offer some comments on our analysis of this paper. 
Note that for initial data in the class satisfying \eqref{oy3.7}-\eqref{2.07} except $u_0\in \dot H^\beta$, the local existence and uniqueness of strong solutions to the Cauchy problem \eqref{a1}-\eqref{h2} has been established recently in \cite{34}. Thus, to extend the strong solutions globally in time and allow the density to vanish initially, one needs global a priori estimates, which are independent of the lower bound of the initial density, of smooth solutions to \eqref{a1}-\eqref{h2} in suitable higher norms. Motivated by the recent work of Li–Xin \cite{Li201977} on the global existence for the 2D compressible Navier-Stokes equations, the key issue here is to derive an upper bound for $\rho$ and $|m|$ that is independent of the lower bound of the initial density.
However, the absence of magnetic dissipation introduces additional difficulties when dealing with the magnetic pressure.

The primary difficulty is to estimate the fluid pressure $P=\rho^\gamma$ and the magnetic pressure $m^2$ simultaneously. Although they enter the momentum equation through the total pressure $P+m^2$, their different powers prevent a direct extension of the argument in \cite[Lemma 3.3]{Li201977} to the full range $\gamma>1$. To overcome this difficulty, thanks to the Bourgain–Brezis divergence estimate (Lemma \ref{lem00}), we can construct a vector field $\Phi$ such that
$$
\div\Phi=P+m^2,\qquad \|\Phi\|_{L^\infty}+\|\na\Phi\|_{L^2}\le C\|P+m^2\|_{L^2},
$$
while the momentum equation can be rewritten as
\begin{align}\label{mini}
-\nabla(P+m^2)=\rho\dot u-\mu\Delta u-(\mu+\lambda)\nabla\operatorname{div}u,
\end{align}
then we can multiply \eqref{mini} by $\Phi$ and integrate the resulting equality over ${\O}$ to obtain
$$
\|P+m^2\|_{L^2}
\le
C\|\rho^{1/2}\dot u\|_{L^2}
+
C\|\nabla u\|_{L^2}.
$$
For details, see \eqref{mgj}. Since $P$ and $m^2$ are nonnegative, the above estimate controls both pressure components for every $\gamma>1$ and does not require a positive lower bound for $\rho$. It is a key ingredient in closing the energy estimates.

The second difficulty arises in deriving the higher integrability estimate for $\int\rho |u|^{2+\nu}dx$. The magnetic $m$ may change sign, and the basic energy estimate does not control $\int |m||u|^2dx$, so the estimates based on the density-weighted kinetic energy cannot be applied directly to $|m|$. 
To overcome this, we use the transport equation for $|m|$ to derive a bound on the spatial moment $\int|m|\bar x^a\,dx$ (see \eqref{mfc}--\eqref{mmfc}). Together with weighted velocity estimates, this bound allows us to control the magnetic and thereby obtain the higher integrability estimate.

Finally, when dealing with the estimation of higher-order derivatives $t^{1/2}\rho^{1/2} u_t\in L_t^\infty L_x^2$, the structure $\nabla m^2=2m\nabla m$ generates $m$ which is sufficient to absorb the positive power weights ${\bar x}^{3(q-2)a/7q}$ generated by $u$. 
This leaves $\nabla m$ unweighted and avoids imposing additional spatial weights on the magnetic gradient.

The rest of the paper is organized as follows: In Section 2, we collect some elementary facts and inequalities that will be used in later analysis. Sections 3 and 4 are devoted to deriving the necessary a priori estimates on strong solutions which are needed to extend the local solution to all time. The main result, Theorem 1.1, is proved in Section 5.

\section{Preliminaries}

In this section, we will recall some  known facts and elementary
inequalities which will be used frequently later.

We begin with the local existence of strong solutions,  whose proof can be obtained by the standard linearization method and the contraction mapping principle. And,  we omit the specific details.

\begin{lemma}\la{th0}
Assume that $(\n_0,u_0,m_0)$ satisfies \eqref{co1} except for the condition $u_0\in \dot H^\beta.$ Then there exist a small time $T_*>0$ and a unique strong solution $(\rho, u,m )$ to the problem \eqref{a1}-\eqref{h2}  in $\O\times(0,T_*)$ satisfying \eqref{1.10} and \eqref{l1.2}.
\end{lemma}


Next, the following well-known Gagliardo-Nirenberg inequality (see \cite{nir}) will be used later.
\begin{lemma}
[Gagliardo-Nirenberg]\la{l1} For  $p\in [2,\infty),q\in(1,\infty), $ and
$ r\in  (2,\infty),$ there exists a constant
$C>0$ which may depend  on $p,q, $ and $r$ such that for   $f\in H^1({\O }) $
and $g\in  L^q(\O )\cap D^{1,r}(\O), $    we have \be
\la{g1}\|f\|_{L^p(\O)}^p\le C \|f\|_{L^2(\O)}^{2}\|\na
f\|_{L^2(\O)}^{p-2} ,\ee  \be
\la{g2}\|g\|_{C\left(\ol{\O }\right)} \le C
\|g\|_{L^q(\O)}^{q(r-2)/(2r+q(r-2))}\|\na g\|_{L^r(\O)}^{2r/(2r+q(r-2))} .
\ee
\end{lemma}

The following weighted $L^p$ bounds for elements of the Hilbert space $  D^{1}(\O)  $ can be found in \cite[Theorem B.1]{L2}.
\begin{lemma} \la{1leo}
   For   $p\in [2,\infty)$ and $\theta\in (1+p/2,\infty),$ there exists a positive constant $C$ such that we have for all $v\in  D^{1,2}(\O),$ \be\la{3h} \left(\int_{\O} \frac{|v|^p}{e+|x|^2}(\log (e+|x|^2))^{-\theta}dx  \right)^{1/p}\le C\|v\|_{L^2(B_1)}+C\|\na v\|_{L^2(\O) }.\ee
\end{lemma}

The following lemma was deduced in \cite{Li201977}, we only state it here without proof.

\begin{lemma}\la{lemma2.6} For $\bar x$   as in \eqref{2.07},
suppose that $\n  \in L^\infty(\O)$ is a   function such that
\be \la{2.12}  0\le \n\le M_1, \quad M_2\le \int_{\b}\n dx ,\quad \n \bar x^\alpha \in L^1(\r^2),\ee
for $ N_*\ge 1 $ and positive constants $   M_1,M_2, $  and   $\al.$  Then, for $r\in [2,\infty),$ there exists a positive constant $C$ depending only on $  M_1, M_2, \alpha,   $ and $ r$  such that
 \be\la{z.1}\left(\int_{\r^2}\n |v |^r dx\right)^{1/r}  \le C  N_*^3  (1+\|\n\bar x^\al\|_{L^1(\r^2)})  \left(  \|\n^{1/2} v\|_{L^2(\r^2)} + \|\na  v \|_{L^2(\r^2)}\right) ,\ee for each $v\in \left.\left\{v\in D^1 (\O)\right|\n^{1/2}v\in L^2(\r^2) \right\}.$

\end{lemma}

Next, we set $ \nabla^{\perp}\triangleq (-\p_2,\p_1),$ $\dot f\triangleq f_t+u\cdot\nabla f,$ where $\dot f$ denoting the material derivative of $f $. We state some elementary estimates which follow from (\ref{g1}) and the standard $L^p$-estimate  for the following elliptic system derived from the momentum equations in (\ref{a1}):
\be\la{h13}
\triangle F = \text{div}(\n\dot{u}),\quad\mu \triangle \o =
\nabla^\perp\cdot(\n\dot{u}) ,
\ee
where $F$ and $\omega$ are as in (\ref{hj1}).

\begin{lemma} \la{le3}
  Let $(\n,u,m)$ be a smooth solution of
   (\ref{a1}).
    Then for   $2\le p<\infty$ there exists a   positive
   constant $C$ depending only on $p,\mu,$ and $\lambda$ such that
\begin{gather}
\|{\nabla F}\|_{L^p} + \|{\nabla \o}\|_{L^p}\le C\norm[L^p]{\n\dot{u}},\label{h19}\\
\norm[L^p]{F} + \norm[L^p]{\o}\le C \norm[L^2]{\n\dot{u}}^{1-2/p }\left(\norm[L^2]{\nabla u} + \norm[L^2]{P }+\|m^2\|_{L^2}\right)^{2/p},\la{h20}
\end{gather}
\begin{equation}
\begin{split}
\norm[L^p]{\nabla u} \le& C \norm[L^2]{\n\dot{u}}^{1-2/p }\left(\norm[L^2]{\nabla u}
   + \norm[L^2]{P }+\|m^2\|_{L^2}\right)^{2/p}\\
   &+ C\norm[L^p]{P}+C\|m^2\|_{L^p}. \la{h18}
\end{split}
\end{equation}
\end{lemma}

The following Bourgain-Brezis divergence estimate follows from \cite[Theorem 5']{BB2007} and will be used later to estimate $\|P+m^2\|_{L^2}.$
\begin{lemma}\label{lem00}
For any $f\in L^2$, there exists a vector field $\Phi\in L^\infty,\ \nabla\Phi\in L^2,$ such that, in the sense of distributions,
\begin{equation}
\div\Phi=f,
\end{equation}
and
\begin{equation}
\norm[L^\infty]{\Phi}+\norm[L^2]{\nabla\Phi}
\le C\norm[L^2]{f},
\end{equation}
where $C>0$ is a positive constant.
\end{lemma}

Next,  in order to get the uniform-in-time upper bound for $\rho$ and $m,$ we need the following Zlotnik inequality.
\begin{lemma}[\cite{zl1}]\la{le1}   Let the function $y$ satisfy
\bnn y'(t)= g(y)+b'(t) \mbox{  on  } [0,T] ,\quad y(0)=y^0, \enn
with $ g\in C(\r)$ and $y,b\in W^{1,1}(0,T).$ If $g(\infty)=-\infty$
and \be\la{a100} b(t_2) -b(t_1) \le N_0 +N_1(t_2-t_1)\ee for all
$0\le t_1<t_2\le T$
  with some $N_0\ge 0$ and $N_1\ge 0,$ then
\bnn y(t)\le \max\left\{y^0,\overline{\zeta} \right\}+N_0<\infty
\mbox{ on
 } [0,T],
\enn
  where $\overline{\zeta} $ is a constant such
that \be\la{a101} g(\zeta)\le -N_1 \quad\mbox{ for }\quad \zeta\ge \overline{\zeta}.\ee
\end{lemma}

Finally,    the following Beale-Kato-Majda type inequality,
which was proved in \cite{bkm} when $\div u\equiv 0,$   will be
used later to estimate $\|\nabla u\|_{L^\infty}$,
$\|\nabla\n\|_{L^2\cap L^q} (q>2)$  and
$\|\nabla m\|_{L^2\cap L^q} (q>2)$.
\begin{lemma}   \la{le9}  For $2<q<\infty,$ there is a
constant  $C(q)$ such that  the following estimate holds for all
$\na u\in L^2(\O)\cap D^{1,q}({\O }),$ \bnn \la{ww7}\ba \|\na
u\|_{L^\infty({\O })}&\le C\left(\|{\rm div}u\|_{L^\infty({\O })}+
\|\o\|_{L^\infty({\O })} \right)\log(e+\|\na^2
u\|_{L^q({\O })}) +C\|\na u\|_{L^2(\O)}+C. \ea\enn
\end{lemma}

\section{A priori estimates (I): lower order estimates}\la{se3}
In this section, we will establish some necessary a priori bounds for smooth solutions to the Cauchy problem (\ref{a1})-(\ref{h2}) to extend the local strong   solution guaranteed by Lemma \ref{th0}. Thus, let $T>0$ be a fixed time and $(\n,u,m)$ be the smooth solution to (\ref{a1})-(\ref{h2})  on
${\r}^2 \times (0,T]$  with smooth initial data $(\n_0,u_0,m_0)$ satisfying (\ref{co1}) and (\ref{2.07}).

Set $\si(t)\triangleq\min\{1,t \}.$  Define
 \begin{eqnarray}
A_1(T)&\triangleq&\sup_{0\le t\le T}\sigma\|\nabla u\|_{L^2}^2 + \int_0^{T} \sigma
 \|\n^{1/2}\dot{u} \|_{L^2}^2  dt,\la{As1}\\[2mm] A_2(T)&\triangleq&\sup_{0\le t\le T}\sigma^2\|\n^{1/2}\dot{u} \|_{L^2}^2 + \int_0^{T}\sigma^2\|\na \dot{u} \|_{L^2}^2dt.\la{As2}
\end{eqnarray}

We have the following key a priori estimates on $(\n,u,m)$.
\begin{proposition}\la{pr1}
Under the conditions of Theorem \ref{th1}, there exists some positive constant $\ve$ depending on $\mu, \lambda, \ga, a, \on, \bar m, \beta, N_0,$ and $M$ such that if $(\n,u,m)$ is a smooth solution of (\ref{a1})-(\ref{h2}) on $\O \times (0,T]$ satisfying
\be\la{z1}
\sup\limits_{\O \times [0,T]}\n\le 2\bar{\n},\quad \sup\limits_{\O \times [0,T]}|m|\le 2\bar{m},\quad A_1(T) + A_2(T) \le 2C_0^{1/2},
\ee
the following estimates hold
\be\la{z2}
\sup\limits_{\O \times [0,T]}\n\le 7\bar{\n}/4,\quad \sup\limits_{\O \times [0,T]}|m|\le 7\bar{m}/4, \quad A_1(T) + A_2(T) 
\le  C_0^{1/2},
\ee
provided $C_0\le \ve.$
\end{proposition}

The proof of Proposition \ref{pr1} will be postponed to the end of this section.

In the following, we will use the convention that $C$ denotes a generic positive constant depending  on $\mu$, $\lambda$, $\gamma$, $a$, $\bar{\n}, \bar m, \beta,$ $N_0,$ and $M$, and  use $C(\al)$ to emphasize that $C$ depends on $\al.$

First, we will prove the  preliminary $L^2$ bounds for  $\nabla u$ and $\n\dot{u}$.

\begin{lemma}\la{le21} Let $(\n,u,m)$ be a smooth solution of
 (\ref{a1})-(\ref{h2}) on $\O \times (0,T]. $ Then there exists a positive constant depending only on $\mu$,  $\lambda$ and $\gamma$ such that
\be \la{a16}
\sup_{0\le t\le T}\int\left(\frac{1}{2}\n|u|^2+\frac{P}{\ga-1} +m^2\right)dx+\ia\int |\na
u|^2 dxdt\le CC_0,
\ee
\begin{equation}\la{h14}
\begin{split}
  A_1(T) \le&  C C_0 +C\sup\limits_{0\le t\le T}(\|P\|_{L^2}^2+\|m^2\|_{L^2}^2)\\
  & + C\int_0^{T}\sigma\int\left(|\nabla u|^3  +(P+m^2)|\nabla u|^2\right)dx dt,
\end{split}
\end{equation}
\begin{equation}\la{h15}
\begin{split}
    A_2(T)
    \le&   C C_0 +CA_1(T) +CA_1^2(T) \\
    &+ C\int_0^{T} \sigma^2 \left(\|\nabla u\|_{L^4}^4 +\|P\|_{L^4}^4+\|m^2\|_{L^4}^4\right) dt.
\end{split}
\end{equation}
\end{lemma}
\begin{proof}
First, applying standard energy estimate gives \eqref{a16} directly.

Next, we will prove \eqref{h14}. Multiplying $(\ref{a1})_2 $ by $ \dot{u}  $ and then integrating the resulting equality over ${\O } $ lead  to
\be\la{m0}\ba
\int\n|\dot{u}|^2dx
&=\mu \int \triangle u\cdot \dot{u}dx + (\mu+\lambda)\int \nabla\text{div}u\cdot \dot{u}dx- \int\dot{u}\cdot\nabla Pdx\\
&\quad-\int\dot u\cdot\na m^2 dx\triangleq\sum_{i=1}^4I_i.
\ea\ee

Similar to the proof of \cite[Lemma 3.2]{Li201977}, we have
\be
\ba\no
\sum_{i=1}^3I_i\le
&\left(\int P\text{div}udx -\frac{\mu }{2} \|\nabla u\|_{L^2}^2-\frac{\lambda+\mu}{2}\|\text{div}u\|_{L^2}^2\right)_t\\
& +C \int P|\na u |^2dx+C\|\na u\|_{L^3}^3.
\ea\ee

Using \eqref{a1}$_3$, we get
\be\no
\ba
I_4&=\int m^2  \div u_tdx+\int m^2  \div (u\cdot \na u)dx\\
&=\left(\int m^2 \div udx\right)_t+2\int(u\cdot \na m+m\div u)m\div udx+\int m^2 \div(u\cdot \na u)dx\\
&\leq\left(\int m^2 \div udx\right)_t+C\int m^2|\nabla u|^2dx.
\ea
\ee

Putting $I_1$-$I_4$ into \eqref{m0} yields
\be\la{n1}
\ba
B'(t)+\int  \n|\dot{u} |^2 dx\le C \int \big(P+m^2\big)|\na u |^2  dx+ C\|\nabla u\|_{L^3}^3,
\ea
\ee
where
\be\ba \la{nv1}
B(t)\triangleq \frac{\mu  }{2}\|\nabla u\|_{L^2}^2 +\frac{ \lambda+\mu}{2}\|\text{div}u\|_{L^2}^2-\int  \text{div}u P dx-\int m^2\text{div}u dx,
\ea\ee
satisfies
\be\la{n2'}\ba
\frac{\mu}{4}\|\na u\|_{L^2}^2-C\|P\|_{L^2}^2-C\|m^2\|_{L^2}^2\le B(t) \le C\|\na u\|_{L^2}^2+C\|P\|_{L^2}^2+C\|m^2\|_{L^2}^2,
\ea\ee
we obtain (\ref{h14}) after integrating \eqref{n1} multiplied by $\si$ over $(0,T)  $ and using   \eqref{a16}  and \eqref{n2'}.

Now, we will prove \eqref{h15}. Applying the operator $\pa/\pa_t+\div(u\cdot~)$ to $ (\ref{a1})_2^j $ and multiplying the resulting equation by $\dot{u}^j$, one obtains after some simple calculations that
\be\label{lv3.40}\ba   
\frac{1}{2}\left(\int \n |\dot{u}^j|^2dx\right)_t
&=\mu\int\dot{u}^j(\Delta u^j_t+\div(u\Delta u^j))dx\\
&\quad+(\mu+\lambda)\int\dot{u}^j(\pa_t\pa_j(\div u)+\div(u\pa_j(\div u)))dx\\
&~~~-\int\dot{u}^j(\pa_jP_t+\div(u\pa_jP))dx\\
&\quad-\int\dot{u}^j(\pa_t\pa_jm^2+\div(u\pa_jm^2))dx\triangleq\sum^4_{i=1}J_i.
\ea\ee

First, following the same arguments as in \cite{H3}, we  have
\be
\ba\no
\sum_{i=1}^{3}J_i&\leq -\frac{3\mu}{4}\|\na\dot{u}\|_{L^2}^2  +C\|\na u\|_{L^4}^4 +C\|P\|_{L^4}^4.
\ea\ee

Next, it follows from  \eqref{a1}$_3$ and \eqref{g1} that
\be
\ba\no
J_4&= 2\int\pa_j\dot{u}^j mm_tdx +\int\pa_i\dot{u}^j u^i\pa_jm^2 dx\\
&=- 2\int\pa_j\dot{u}^j m (u\cdot \na m+m\div u)dx-\int\big(\pa_i\pa_j\dot{u}^j  u^im^2+\pa_i\dot{u}^j \pa_ju^im^2 \big)dx\\
&\leq C\int|\na \dot{u}| |\na u||m|^2dx\\
&\leq \frac{\mu}{4}\int|\na \dot{u}|^2dx+C\int|\na u|^4dx+C\int m^8dx .
\ea\ee

Putting $J_1$-$J_4$ into \eqref{lv3.40} yields
\be\label{lv3.46}\ba
\left(\int \n |\dot{u}|^2dx\right)_t+\mu\int|\na\dot{u}|^2dx
\leq C \|\na u\|^4_{L^4}+C \|P\|^4_{L^4}+C \|m^2\|^4_{L^4},
\ea\ee
which multiplied by $\sigma^2$ gives \eqref{h15} and completes the proof of Lemma \ref{le21}.
\end{proof}

that the pressure and the magnetic field
decay in time.
\begin{lemma}\la{le5}
Let $(\n,u,m)$ be a smooth solution  of (\ref{a1})-(\ref{h2}) on $\O \times (0,T] $ satisfying (\ref{z1}). Then there exists a positive constant $C(\on, \bar m)$ depending only  on $\mu,$  $\lambda,$  $\gamma,$ $\on$ and $\bar m$ such that
\be\la{h27}
A_1(T)+A_2(T)+\int_0^T\bigl(\si\|P\|_{L^2}^2+\si\|m^2\|_{L^2}^2\bigr)dt\le C(\on, \bar m) C_0 .
\ee
\end{lemma}
\begin{proof}
First, it follows from (\ref{h18}), \eqref{z1} and \eqref{a16}  that
\be
\ba\la{h99}
&\int_0^{T}\sigma^2\left(\|\na u\|_{L^4}^4+\|P\|_{L^4}^4+\|m^2\|_{L^4}^4 \right)dt\\
&\le C\int_0^{T}\si\|\n\dot u\|_{L^2}^2\left(\si\|\na u\|_{L^2}^2+\si\|P\|_{L^2}^2+\si\|m^2\|_{L^2}^2\right)dt\\
&\quad+C\int_0^{T}\sigma^2\|P\|_{L^4}^4dt+C\int_0^T\si^2\|m^2\|_{L^4}^4dt\\
&\le C(\on,\bar m)\left(A_1(T)+C_0 \right)\int_0^{T}\sigma\|\n^{1/2}\dot u\|_{L^2}^2dt\\
&\quad+C(\on)\int_0^{T}\sigma^2\|P\|_{L^2}^2dt+C(\bar m) \int_0^{T}\sigma^2 \|m^2\|_{L^2}^2dt\\
&\le C(\on, \bar m) C_0+C(\on,\bar m)\int_0^{T}\sigma^2\|P+m^2\|_{L^2}^2dt.
\ea
\ee

Next, to estimate the last term on the right-hand side of \eqref{h99}, we rewrite $(\ref{a1})_2$ as
\be\label{pm21}
-\nabla(P+m^2)=\n  \dot u -\mu\Delta u-(\mu+\lambda)\nabla\div u.
\ee
Applying Lemma \ref{lem00} to $P+m^2\in L^2(\O)$, there exists a vector field $\Phi:\O\to\r^2$ such that
\be\label{pm22}
\div\Phi=P+m^2,\qquad \|\Phi\|_{L^\infty}+\|\na\Phi\|_{L^2}\le C\|P+m^2\|_{L^2}.
\ee
By direct calculation and combining with \eqref{pm21}-\eqref{pm22}, we obtain
\begin{equation}
\begin{split} \la{mgj}
\|P+m^2\|_{L^2}^2
&= \int(P+m^2)(P+m^2)dx=-\int\nabla(P+m^2)\cdot\Phi dx\\
&=\int\n\dot u\cdot\Phi\,dx+\mu\int\na u:\na\Phi\,dx
 +(\mu+\lambda)\int\div u\,\div\Phi\,dx\\
&\le \|\n\dot u\|_{L^1}\|\Phi\|_{L^\infty}
 +C\|\na u\|_{L^2}\|\na\Phi\|_{L^2}\\
&\le C\left(\|\n\|_{L^1}^{1/2}\|\n^{1/2}\dot u\|_{L^2}
 +\|\na u\|_{L^2}\right)\|P+m^2\|_{L^2}\\
&=C\left(\|\n^{1/2}\dot u\|_{L^2}+\|\na u\|_{L^2}\right)\|P+m^2\|_{L^2}.
\end{split}
\end{equation}
where in the last inequality, one has used
\be \la{mr3}\int\n dx =\int \n_0dx =1.\ee

From \eqref{mgj}, we have
\be \la{new1}\ba \|P+m^2\|_{L^2}  \le   C  \| \n^{1/2}  \dot u \|_{L^2} + { C}\|\na u\|_{L^2},\ea \ee
which, along with  \eqref{h14},   \eqref{h15}, \eqref{h99}, \eqref{a16}, and \eqref{z1} gives   \be\la{h28}\ba
A_1(T)+A_2(T)\le& C(\on,\bar m)C_0 +C(\on,\bar m)\int_0^{T}  \sigma\|\nabla u\|_{L^3}^3 dt. \ea\ee

 Finally, on the one hand,  one deduces from (\ref{h18}),  \eqref{g1}, (\ref{a16})  and (\ref{z1})  that 
 \be\la{h34} \ba
  \int_0^{\si(T)} \sigma\|\nabla u\|_{L^3}^3 dt
&  \le C \int_0^{\si(T)} \sigma\|\n^{1/2} \dot u \|_{L^2} \left(\|\na u\|_{L^2}^2+\|P\|_{L^2}^2 +\|m^2\|_{L^2}^2\right)  dt \\
&\quad+C\int_0^{\si(T)}(\si\|P\|_{L^3}^3+ \si\|m^2\|_{L^3}^3)dt \\
&  \le C A_2^{1/2}(\si(T))\int_0^{\si(T)}  \left(\|\na u\|_{L^2}^2+\|P\|_{L^2}^2+\|m^2\|_{L^2}^2 \right)dt\\
&\quad+C\int_0^{\si(T)}(\|P\|_{L^3}^3+\|m^2\|_{L^3}^3) dt\\
& \le  C(\on,\bar m)C_0.\ea \ee
%
 On the other hand, H\"older's inequality, \eqref{h99},  \eqref{z1}, and \eqref{new1} imply \be\la{h33} \ba
 \int_{\si(T)}^{T}  \|\nabla u\|_{L^3}^3  dt
&\le
\de\int_{\si(T)}^{T}    \|\nabla u\|_{L^4}^4 dt + C(\de) \int_{\si(T)}^{T} \|\nabla u\|_{L^2}^2 dt \\
& \le  \de C(\on,\bar m) A_1(T)+  C(\de)C(\on, \bar m) C_0.
\ea \ee
Putting   \eqref{h34}  and \eqref{h33} into \eqref{h28} and choosing $\de$ suitably small leads to \bnn A_1(T)+A_2(T)\le C(\on,\bar m)C_0,\enn  which together with \eqref{new1} and \eqref{a16} gives \eqref{h27} and finishes  the proof of Lemma \ref{le5}.
\end{proof}


Next, we derive the rates of decay for $\na u$, $P$ and $m$, which are  essential to obtain the uniform  (in time) upper bound of the $\rho$ and $m$ for large time.
\begin{lemma}\la{ly1}
For $p\in [2,\infty),$ there exists a positive constant $C(p,\bar\rho,\bar m)$ depending only on $p, \mu,$  $\lambda,$  $\gamma,$ $\on$ and $\bar m$ such that,  if  $(\n,u,m)$ is a smooth solution  of \eqref{a1}-\eqref{h2} on $\O \times (0,T] $ satisfying \eqref{z1}, then
\be \la{ly8} \ba &\sup\limits_{\si(T)\le t\le T}\left[t^{p-1}\left(\|\na u\|_{L^p}^p + \|P\|^p_{L^p}+ \|m^2\|^p_{L^p}\right)+t^2\|\n^{1/2}\dot u\|_{L^2}^2 \right] \le C(p,\on,\bar m)C_0.\ea\ee
\end{lemma}
\begin{proof}
First, the pressure  $P$ satisfies \be
\la{a95}P_t+u\cdot\nabla P+\ga P{\rm div}u=0 .\ee For $p\ge 2,$ multiplying \eqref{a95} by $p P^{p-1}$ and integrating
the resulting equality over ${\O } ,$ one gets after using ${\rm
div}u=\frac{1}{2\mu+\lambda}\left(F+P+m^2\right)$ that
\be\la{vv1}
\ba
\left(\| P\|_{L^p}^p \right)_t+ \frac{p\ga-1}{2\mu+\lambda}\|P\|_{L^{p+1}}^{p+1}
&=- \frac{p\ga-1}{2\mu+\lambda}\int P^p\left(F+m^2\right)dx \\
&\leq- \frac{p\ga-1}{2\mu+\lambda}\int P^pFdx \\
&\le \frac{p\ga-1}{2(2\mu+\lambda)}\|P\|_{L^{p+1}}^{p+1}+C(p)  \|F\|_{L^{p+1}}^{p+1},
\ea\ee
where we used $m^2\geq 0$, which together with \eqref{h20} and \eqref{g1} gives
\be\la{a96}\ba
\left(\| P\|_{L^p}^p \right)_t+ \frac{p\ga-1}{2(2\mu+\lambda)}\|P\|_{L^{p+1}}^{p+1}
\le C(p) \left(\|\na u\|_{L^2}^2+\|P\|_{L^2}^2+\|m^2\|_{L^2}^2\right)\|\n \dot u\|_{L^2}^{p-1}.
\ea\ee

Next, $m^2$ satisfies
\be\la{aa95}
(m^2)_t+u\cdot\nabla m^2+2 m^2{\rm div}u=0.
\ee
By similar calculations, we can obtain
\be\la{aa96}\ba
 \left(\|m^2\|_{L^p}^p \right)_t+ \frac{2p-1}{2(2\mu+\lambda)}\|m^2\|_{L^{p+1}}^{p+1}
    \le C(p) \left(\|\na u\|_{L^2}^2+\|P\|_{L^2}^2+\|m^2\|_{L^2}^2\right)\|\n \dot u\|_{L^2}^{p-1}.\ea\ee

Combining \eqref{a96} and \eqref{aa96} yields
\be\la{aa96a}\ba
 &\left(\| P\|_{L^p}^p+\|m^2\|_{L^p}^p \right)_t+ \frac{p-1}{2(2\mu+\lambda)}\left(\|P\|_{L^{p+1}}^{p+1}+ \|m^2\|_{L^{p+1}}^{p+1}\right)\\
    &\le C(p) \left(\|\na u\|_{L^2}^2+\|P\|_{L^2}^2+\|m^2\|_{L^2}^2\right)\|\n \dot u\|_{L^2}^{p-1}.
    \ea\ee

In particular, choosing $p=2$  in \eqref{aa96a} shows
\be\ba\la{aaa9z6}
&\left(\| P\|_{L^2}^2+\| m^2\|_{L^2}^2 \right)_t+ \frac{1}{2(2\mu+\lambda)}\left(\|P\|_{L^3}^3 + \|m^2\|_{L^3}^3 \right) \\
&\leq \de \|\n^{1/2}\dot u\|_{L^2}^2+C(\de)\left(\|\na u\|_{L^2}^4+\|P\|_{L^2}^4+\|m^2\|_{L^2}^4\right).
\ea\ee

Next, it follows from \eqref{n1}, \eqref{h18},  \eqref{g1}, and \eqref{a16} that
\begin{equation}   \la{ly11}\begin{split}
&B'(t)+\int \n|\dot{u} |^2 dx\\
& \le C\| P\|_{L^3}^3+ C\|m^2\|_{L^3}^3+C  \|\nabla u\|_{L^3}^3 \\
& \le C_1  \|P\|_{L^3}^3+ C_1\|m^2\|_{L^3}^3+ C \|\n\dot u\|_{L^2}\left(\|\na u\|_{L^2}^2+\|P\|_{L^2}^2+\|m^2\|_{L^2}^2\right)\\
& \le C_1 \|P\|_{L^3}^3+ C_1\|m^2\|_{L^3}^3+\delta\|\n^{1/2}\dot u\|_{L^2}^2+C(\on,\bar m )\left(\|\na u\|_{L^2}^4+\|P\|_{L^2}^4+\|m^2\|_{L^2}^4\right).
\end{split}
\end{equation}

Choosing $ C_2\ge\max\{ 2+  4(2\mu+\lm)(C_1+1)\}$
so large that
\be\la{ly12}\ba
&\frac{\mu}{4}\|\na u\|_{L^2}^2+\|P\|_{L^2}^2 + \|m^2\|_{L^2}^2\\&\le B(t)+C_2\|P\|_{L^2}^2+C_2\|m^2\|_{L^2}^2 \le C\|\na u\|_{L^2}^2+C\|P\|_{L^2}^2+C\|m^2\|_{L^2}^2,
\ea\ee
and adding \eqref{aaa9z6} multiplied by $C_2$  to \eqref{ly11}, we obtain after choosing $\de$ suitably small that
\begin{equation} \la{h29}\begin{split} &2\left( B(t)+C_2 \| P\|_{L^2}^2+C_2\| m^2 \|_{L^2}^2 \right)_t +\int  \left(\n|\dot{u} |^2 + P^3+m^6  \right)dx \\
& \le
C \| P\|_{L^2}^4+C\|m^2\|_{L^2}^4 + C  \|\na u\|_{L^2}^4 ,
\end{split}
\end{equation}
multiplying \eqref{h29} by $t$ and applying Gr\"onwall's inequality, together with \eqref{ly12}, \eqref{h27},  \eqref{a16}, and \eqref{z1} yields
 \be \la{h31} \ba & \sup\limits_{
\si(T)\le t\le T}t\left( \| \na u\|_{L^2}^2+ \| P\|_{L^2}^2+\| m^2 \|_{L^2}^2
 \right) \\& + \int_{\si(T)}^Tt \int  \left(\n|\dot{u} |^2+ P^3 +m^6 \right)dxdt\le C(\on,\bar m)C_0. \ea\ee

Then, multiplying \eqref{lv3.46} by $t^2 $ shows
\be\la{zo1}
\ba
&\left(t^2\int\n|\dot{u}|^2dx \right)_t + { \mu}t^2 \int|\nabla\dot{u}|^2dx \\
&\le 2t\int\n|\dot{u}|^2dx+C(\bar\rho,\bar m)t^2\|\rho\dot u\|_{L^2}^2(\|\nabla u\|_{L^2}^2+\|P\|_{L^2}^2+\|m^2\|_{L^2}^2)\\
&\quad+C_3(\bar\rho,\bar m)t^2(\|P\|_{L^4}^4+\|m^2\|_{L^4}^4).
\ea
\ee

Choosing $p=3$ in \eqref{aa96a}, we add  \eqref{aa96a}  multiplied by $2(2\mu+\lambda)(C_3+1)t^2$ to \eqref{zo1} and obtain that for $t\in [\si(T),T]$
\be\la{lbq-jia25'} \ba
& \left(t^2\int\n|\dot{u}|^2dx+ 2(2\mu+\lambda)( C_3 +1)t^2\big(\| P\|_{L^3}^3+ \| m^2\|_{L^3}^3\big)\right)_t\\
&\quad+ {\mu}t^2\|\nabla\dot{u}\|^2_{L^2}+t^2\|P\|^4_{L^4}+t^2\|m^2\|^4_{L^4}\\
&\leq Ct\left(\|\n^{1/2}\dot{u}\|^2_{L^2}+\|P \|^3_{L^3}+\|m^2 \|^3_{L^3}\right)\\
&\quad+C(\bar\rho,\bar m)t^2\|\n^{1/2}\dot{u}\|^2_{L^2}(\|\nabla u \|_{L^2}^2+\|P \|_{L^2}^2+\|m^2 \|^2_{L^2}),
\ea\ee
which, combined with \eqref{h31},  and   \eqref{z1} yields
\be \la{ly9} \ba
& \sup\limits_{\si(T)\le t\le T}t^2\int  \left(\n|\dot{u} |^2+ P^3+m^6  \right)dx \\
&\quad+  \int_{\si(T)}^Tt^2 \left(\|\nabla\dot{u}\|^2_{L^2}+\|P\|_{L^4}^4+\|m^2\|_{L^4}^4\right)dt \le C(\on,\bar m)C_0.
 \ea\ee

Finally, we claim that for $k=1,2,\cdots,$
\be \la{ly10}
\sup\limits_{\si(T)\le t\le   T}t^{k}(\|P\|_{L^{k+1}}^{k+1}+\|m^2\|_{L^{k+1}}^{k+1})+\int_{\si(T)}^Tt^k(\|P\|_{L^{k+2}}^{k+2}+\|m^2\|_{L^{k+2}}^{k+2})dt\le C(k,\on,\bar m)C_0,
\ee
which together with \eqref{h18}, \eqref{h31}, and \eqref{ly9} gives \eqref{ly8}.   We shall prove \eqref{ly10} by induction. In fact, for $k=1,$ \eqref{ly10} holds  due to \eqref{h31}. Assume that \eqref{ly10} holds for $k=n,$ that is,
\be \la{ly16}
\sup\limits_{\si(T)\le t\le T}t^{n}(\|P\|_{L^{n+1}}^{n+1}+\|m^2\|_{L^{n+1}}^{n+1})+\int_{\si(T)}^Tt^n(\|P\|_{L^{n+2}}^{n+2}+\|m^2\|_{L^{n+2}}^{n+2})dt\le C(n,\on,\bar m)C_0.
\ee
Multiplying \eqref{aa96a} (where  $p=n+2$)  by $t^{n+1} ,$ one obtains after  using \eqref{ly9}
\be\la{zo8}\ba &
  \left(t^{n+1}\| P\|_{L^{n+2}}^{n+2}+t^{n+1}\| m^2\|_{L^{n+2}}^{n+2} \right)_t+\frac{n+1}{2(2\mu+\lambda)} t^{n+1}(\|P\|_{L^{n+3}}^{n+3}+\|m^2\|_{L^{n+3}}^{n+3})\\
&\le C (n,\on) t^{n}(\| P\|_{L^{ n+2} }^{n+2}+\| m^2\|_{L^{ n+2} }^{n+2})+C(n,\on)C_0\left(\|\na u\|^2_{L^2}+ \|P \|^2_{L^2}+\|m^2\|^2_{L^2}\right).
\ea\ee
Integrating \eqref{zo8} over $[\si(T),T]$ together with \eqref{ly16} and  \eqref{h27} shows that \eqref{ly10} holds for $k=n+1.$   By induction, we obtain \eqref{ly10} and finish the proof of Lemma \ref{ly1}.
\end{proof}

Next, the following Lemma \ref{lemma2.7} combined with Lemma \ref{lemma2.6} will be useful to estimate the $L^p$-norm of $\n\dot u$ and obtain the uniform  (in time) upper bound of the $\rho$ and $m$ for large time.

\begin{lemma}\la{lemma2.7}
Let $(\n,u,m)$ be a smooth solution  of (\ref{a1})-(\ref{h2}) on $\O \times (0,T] $ satisfying (\ref{z1}). Then there exists a positive constant $C$ depending only  on $\mu,$  $\lambda,$  $\gamma, a,$ $\on$, $N_0$, $\bar m$ and $M$ such that
\be \la{o4.7} \sup_{0\le t\le \sigma(T)}\int|m| \bar x^a  dx\le  C(\bar\rho,\bar m,M).\ee
\end{lemma}
\begin{proof}
First, multiplying $\eqref{a1}_1$ by $(1+|x|^2)^{1/2} $ and integrating the resulting equality over $\O$ by parts (see \cite[(3.40)]{Li201977} also) we can obtain
\be \la{o3.7}
\sup_{0\le s\le t}\int\n  (1+|x|^2)^{1/2}  dx\le  C(M)(1+t).
\ee

Next, we recall the following inequality, proved in \cite[Lemma 3.5]{Li201977}. For any $\al>1,$ there
exists a positive constant $N_1$ depending only on $\al,$ $N_0,$ and $M$ such that for all $t\in (0,T],$
\be\la{uq2}
\int_{B_{N_1(1+t)\log^\al(e+t)}}\n(x,t) dx \ge  \frac14.
\ee

Then, it follows from  \eqref{3h}, \eqref{uq2}, and the Poincar\'e-type inequality \cite[Lemma 3.2]{F2} that for  $s>2, \delta\in(0,1], $ and $t\in [0,\si(T)],$
\be \la{ljo1}
\|u\bar x^{-\delta}\|_{L^{s/\delta}}\le C(s,\delta,\on,N_0,M)\|\n^{1/2}u\|_{L^2}+C(s,\delta,\on,N_0,M)\|\na u\|_{L^2}.
\ee

Finally, observing that $|m|$ satisfies
\begin{align}\label{mfc}
(|m|)_t+\div(|m|u)=0,
\end{align}
we multiply $\eqref{mfc}$ by $\bar x^{a}$ and integrate the resulting equality over $\O $ leads to
\begin{equation}
\begin{split}\label{mmfc}
\frac{d}{dt}\int|m| \bar x^a  dx &\le C\int |m| |u|\bar x^{a -1} \log^2 (e+|x|^2) dx\\ 
&\le C \||m|\bar x^{a-1+8/(8+a)}\|_{L^{(8+a)/(7+a)}}\|u\bar x^{-4/(8+a)}\|_{L^{8+a}}\\
&\le C(\bar\rho,\bar m,N_0,M)\left(1+\int|m| \bar x^a  dx \right)\left(\|\rho^{1/2}u\|_{L^2}+\|\na u\|_{L^2}\right).
\end{split}
\end{equation}
which combining with \eqref{a16} gives \eqref{o4.7}.
\end{proof}

Next, to obtain the upper bound of the $\rho$ and $m$ for small time, we still need the following lemma.

\begin{lemma}\la{zc1} Let $(\n,u,m)$ be a smooth solution of (\ref{a1})-(\ref{h2})   on $\O \times (0,T] $ satisfying (\ref{z1}) and the assumptions in Theorem \ref{th1}. Then there
exists a positive constant  $K $   depending only on $\mu ,  \lambda ,   \ga ,  a ,  \on, \bar m, \beta,$ $N_0,$ and $M$  such
that
   \be\la{uv1}  \sup_{0\le t\le \si(T)}t^{1-\beta}\|\na
u\|_{L^2}^2+\int_0^{\si(T)}t^{1-\beta}\int\n|\dot u|^2dxdt\le
K(\on,\bar m,M). \ee

\end{lemma}
\begin{proof}
First, set 
$$\nu= \min\left\{\frac{\mu^{1/2}}{ 2(1+2\mu+\lm)^{1/2} },\,\frac{ \beta}{1-\beta}\right\}\ for\ \beta\in(0,1),\quad \nu=\frac{\mu^{1/2}}{ 2(1+2\mu+\lm)^{1/2} }\ for\ \beta=1.
$$ If $\beta\in (0,1),$   Sobolev's inequality implies
\be \la{zz.1}\ba\int \n_0|u_0|^{2+\nu}dx &\le  \int \n_0|u_0|^{2 }dx +  \int \n_0 |u_0|^{2/(1-\beta)}dx\\ &\le C(\on)+C(\on)\|u_0\|_{\dot H^\beta}^{2/(1-\beta)}\le C(\on,M).\ea\ee
For $\beta=1,$ one obtains from \eqref{z.1} that
\be  \la{zz.2}\ba\int \n_0|u_0|^{2+\nu}dx &\le C(\on)\left( \int \n_0|u_0|^{2 }dx +  \int   |\na u_0|^2dx\right)^{(2+\nu)/2} \le C(\on,M).\ea\ee

Then, multiplying $(\ref{a1})_2$ by $(2+\nu)|u|^\nu u$  and   integrating the resulting equation over $\O$ lead  to
\be\ba\label{nugj}
& \frac{d}{dt}\int \n |u|^{2+\nu}dx+ (2+\nu) \int|u|^\nu (\mu |\na u|^2+(\mu+\lambda) (\div u)^2) dx \\
& \le  (2+\nu)\nu  \int (\mu+\lambda)|\div u||u|^\nu |\na u|dx +C  \int \n^\gamma |u|^\nu |\na u|dx+C\int m^2|u|^\nu|\na u|dx  \\
& \le \frac{2+\nu}{2}\int (\mu+\lambda)(\div u)^2|u|^\nu  dx +\frac{(2+\nu)\mu}{4} \int  |u|^\nu |\na u|^2 dx \\
&\quad + C \int \n |u|^{2+\nu} dx+C\int \n^{(2+\nu)\gamma-\nu/2}dx+C\int m^4 |u|^\nu dx.
\ea\ee
It follows from \eqref{a16}, \eqref{o4.7}, and \eqref{ljo1}  that
\be\la{lbq-jia2}\ba
&\int_0^{\si(T)} \int m^4 |u|^\nu dxdt\\
&\leq C\int_0^{\si(T)}\|m^4{\bar x}^{(2-\nu)a/2}\|_{L^{2/(2-\nu)}}\||u|^\nu\bar x^{(\nu-2)a/2}\|_{L^{2/\nu}}dt\\
&\leq  C(\on,N_0,M)\int_0^{\si(T)} \left(\int |m|^{8/(2-\nu)}{\bar x}^a dx\right)^{(2-\nu)/2}\left(\|\rho^{1/2}u\|_{L^2}+\|\na u\|_{L^2} \right)^{\nu}dt\\
&\leq  C(\on,\bar m,N_0,M)\int_0^{\si(T)} \left(\int |m|{\bar x}^a dx\right)^{(2-\nu)/2}\left(\|\rho^{1/2}u\|_{L^2}+\|\na u\|_{L^2} \right)^{\nu}dt\\
&\le C(\on,\bar m,N_0,M)\int_0^{\si(T)}\left(1+\|\rho^{1/2}u \|_{L^2}^2+ \|\na u\|_{L^2}^2\right)dt\le C(\on,\bar m,M),
\ea\ee
where we used $\nu\in(0,2/3)$ and
\begin{align*}
\||u|^\nu\bar x^{(\nu-2)a/2}\|_{L^{2/\nu}}&\leq C\left(\|u\bar x^{-1}\|_{L^{(2-\nu)a/\nu}}
\|\bar x^{1+{(\nu-2)a}/{(2\nu)}}\|_{L^\frac{2a(2-\nu)}{2a-\nu a-2\nu}}\right)^\nu\\
&\leq C(\bar\rho,N_0, M)\left(\|\rho^{1/2}u\|_{L^2}+\|\na u\|_{L^2} \right)^{\nu}.
\end{align*}

Then, applying Gr\"onwall's inequality to \eqref{nugj}, and using \eqref{zz.1}, \eqref{zz.2}, and \eqref{lbq-jia2}, we obtain 
\be\la{jan2}
\sup_{0\le t\le  \si(T) }\int \n |u|^{2+\nu}dx\le C(\on,\bar m,N_0,M ) .
\ee

Next, as in \cite{hof2002}, for the linear differential operator $L$ defined by
\be\ba
(Lw)^j&\triangleq \n w^j_{t}+ \n u\cdot\na w^j
-(\mu\Delta w^j+(\mu+\lambda) \pa_j\div w )\no
&=\n\dot w^j
-(\mu\Delta w^j+(\mu+\lambda) \pa_j\div w ),\quad j=1,2  ,
\ea\ee
let $w_1$ and $w_2$  be the solutions to:
\be \la{sas2}
Lw_1 =0 ,\quad w_1(x,0)=w_{10}(x),
\ee
and
\be\la{sas3}
Lw_2 = -\nabla P(\n)-\nabla m^2,\quad w_2(x,0)=0,
\ee
respectively. A straightforward energy estimate of \eqref{sas2}  shows  that:
\be\la{sas4}
\sup_{0\le t\le\si(T)}\int \n|w_1|^2dx+\int^{\si(T)}_{0}
\int|\nabla w_1|^2dxdt\leq C(\overline{\n})\int |w_{10}|^2dx.
\ee

Then, multiplying (\ref{sas2}) by
$ w_{1t} $ and integrating the resulting equality over $\O $  yield  that  for $t\in (0,\si(T)],$
 \be\la{sas5}\ba& \frac{1}{2}\left( \mu\|\na {w_1}\|_{L^2}^2
 +(\mu+\lambda)\|\div {w_1}\|^2_{L^2} \right)_t
 +\int\n|\dot {w_1}|^2dx\\&=\int \n\dot {w_1}(u\cdot\na {w_1})dx
   \\&\le  C(\on )\|\n^{1/2}\dot w_1\|_{L^2} \|\n^{1/(2+\nu)}u\|_{L^{2+\nu}}\| \na^2  w_1\|_{L^2}^{2/(2+\nu)} \| \na  w_1\|_{L^2}^{\nu/(2+\nu)}
  \\&\le \frac12\int\n|\dot {w_1}|^2dx+C(\on,\bar m,M)  \| \na  w_1\|_{L^2}^2,
  \ea\ee
  where in the last inequality we have used \eqref{jan2} and the following simple fact:
$$ \|\na^2w_1\|_{L^2}\le C\|\n\dot w_1\|_{L^2}, $$
  due to     the standard $L^2$-estimate of the  elliptic  system  (\ref{sas2}).
 Gr\"onwall's inequality  together with (\ref{sas5}) and (\ref{sas4})  gives
\be \label{uu1} \sup_{0\le t\le \si(T)}\|\nabla
w_1\|_{L^2}^2+\int_0^{\si(T)}\int\n|\dot {w_1}|^2dxdt\leq C(\on,\bar m,M)\|\na
w_{10}\|_{L^2}^2,\ee and \be  \label{uu2} \sup_{0\le t\le
\si(T)}t\|\nabla w_1\|_{L^2}^2+\int_0^{\si(T)}t\int\n|\dot
{w_1}|^2dxdt\leq C(\on,\bar m,M)\| w_{10}\|_{L^2}^2.\ee

Since the solution operator $w_{10}\mapsto w_1(\cdot,t)$ is linear,
by the   standard Stein-Weiss interpolation argument (see \cite{bl}),
one can deduce from (\ref{uu1}) and (\ref{uu2}) that for any
$\theta\in [\beta,1],$\be \label{uu4} \sup_{0\le t\le
\si(T)}t^{1-\theta}\|\nabla
w_1\|_{L^2}^2+\int_0^{\si(T)}t^{1-\theta}\int\n|\dot
{w_1}|^2dxdt\leq C(\on,\bar m,M)\| w_{10}\|_{\dot H^\theta}^2,\ee with a uniform
constant $C$ independent of $\theta.$

Finally, we estimate $w_2.$ In a similar way to the proof of
(\ref{h19}) and (\ref{h18}), one obtians that
\be\la{sas6}\|\na\big((2\mu+\lambda)\div w_2-P- m^2\big)\|_{L^2}+\|\na (\na^\perp\cdot w_2)\|_{L^2}\le C\|\n  \dot
w_2 \|_{L^2} ,
\ee
and that for $p\ge 2,$
\be \la{sax6}\ba
\|\na w_2\|_{L^p}&  \le C (\|(2\mu+\lambda)\div w_2-P- m^2\|_{L^p}+C\|P\|_{L^p}+C\|m^2\|_{L^p}+\|\na^\perp\cdot w_2\|_{L^p}   \\
&\le \de\|\n  \dot
w_2 \|_{L^2}  +C(\on,\bar m,p,\de)\|\na w_2\|_{L^2}  +C(\on,\bar m,p,\de)C_0^{1/p} .
\ea\ee
Multiplying (\ref{sas3}) by $w_{2t} $ and  integrating the resulting equation over $\O  $   yields that for $t\in (0,\si(T)],$
\be \la{xac6}\ba
& \frac{1}{2}\left( \mu\|\na {w_2}\|_{L^2}^2+(\mu+\lambda)\|\div {w_2}\|^2_{L^2} -2\int P\div w_2dx -2\int m^2\div w_2dx\right)_t
+\int\n|\dot {w_2}|^2dx\\
&=\int \n\dot {w_2}(u\cdot\na{w_2})dx-\int \big(P_t+(m^2)_t\big)\div w_2dx \\
&\le  C(\on)\|\n^{1/2}\dot w_2\|_{L^2}\|\n^{1/(2+\nu)}u\|_{L^{2+\nu}}\| \na  w_2\|_{L^{2(2+\nu)/\nu}} -\int \big(P_t+(m^2)_t\big)\div w_2dx \\
&\le C(\on,M)\de\|\n^{1/2}\dot w_2\|_{L^2}^2 +C(\de,\on,\bar m,M)\left(\|\na w_2\|_{L^2}^2 +\|\na u\|_{L^2}^2+1\right),
\ea\ee
where in the last inequality we have used (\ref{sax6}), \eqref{jan2},  and the following simple fact:
\be\ba
&-\int \big(P_t+(m^2)_t\big)\div w_2dx     \\
&= -\frac{1}{2\mu+\lm}\int (P+m^2)u\cdot\na ((2\mu+\lm)\div w_2-P-m^2)dx \\
&\quad+\frac{1}{2(2\mu+\lm)}\int (P+m^2)^2\div u dx+\int [(\gamma-1)P+m^2]\div u\div w_2dx \\
&\le C\|Pu\|_{L^2}\|\n\dot w_2\|_{L^2}+C\|m^2{\bar x}^{\alpha}\|_{L^3}\|u{\bar x}^{-\alpha}\|_{L^6}\|\n\dot w_2\|_{L^2}+ C(\|P^2\|_{L^2}+\|m^4\|_{L^2})\|\na u\|_{L^2}\\
&\quad+ C\|\na u\|_{L^2}\|\nabla w_2\|_{L^2}  \\
&\le \de\|\n^{1/2}\dot w_2\|_{L^2}^2 +C(\de,\on,\bar m,M )\left( \|\na w_2\|_{L^2}^2+\|\na u\|_{L^2}^2+1\right),
\ea\ee
where $\alpha=\min\{a/3,1/2\}$, and where we used \eqref{a95} and \eqref{sas6}. Gr\"onwall's inequality together with (\ref{xac6})  gives
\be  \label{uu3}
\sup_{0\le t\le \si(T)}\|\nabla w_2\|_{L^2}^2+\int_0^{\si(T)}\int\n|\dot {w_2}|^2dxdt\leq
C(\on,\bar m,M).
\ee
Taking $w_{10}=u_0 $ so that $w_1+w_2=u,$
we then derive (\ref{uv1}) from (\ref{uu4}) and (\ref{uu3}) directly.
Thus, we finish the proof of Lemma \ref{zc1}.
\end{proof}

We now proceed to derive a uniform (in time) upper bound for the $\rho$ and $m$, which turns out to be the key to obtain all the higher order estimates.
We will use an approach motivated by   \cite{lx,HuangLiXin2012}).
\begin{lemma}\la{le7}
There exists a positive constant $\ve_0=\ve_0(\on, \bar m, M)$ depending on $\mu, \lambda, \ga, a, \on, \bar m,$ $\beta,$ $N_0,$ and $M$ such that, if $(\n,u,m)$ is a smooth solution  of (\ref{a1})-(\ref{h2}) on $\O \times (0,T]$ satisfying (\ref{z1}) and the assumptions in Theorem \ref{th1}, then
\be\la{lv102}
\sup_{0\le t\le T}\|\n(t)\|_{L^\infty}\le\frac{7\bar \n }{4},\quad \sup_{0\le t\le T}\|m(t)\|_{L^\infty}\le\frac{7\bar m }{4},
\ee
provided
$C_0\le \ve_0\triangleq\min\{\ve_1,\ve_2,\ve_3,\ve_4\}.$
\end{lemma}
\begin{proof}
First, we rewrite the equation of the mass conservation
$(\ref{a1})_1$ as \be \la{z.3} D_t \n=g(\n)+b'(t), \ee where \bnn
D_t\n\triangleq\n_t+u \cdot\nabla \n ,\quad
g(\n)\triangleq-\frac{\n^{\ga+1 }}{2\mu+\lambda}  ,
\quad b(t)\triangleq-\frac{1}{2\mu+\lambda} \int_0^t\n \left(F+m^2\right)dt. \enn

It follows from \eqref{h19}, \eqref{uq2}, \eqref{o3.7}, and \eqref{z.1} that for $t>0$ and $p\in [2,\infty),$
\be\la{z.2}\ba   \|\na F\|_{L^p}
&\le C(p)\|\n \dot u \|_{L^p} \\
&\le C(p,\on,\bar m,M)(1+t)^5\left(\|\n^{1/2}\dot u \|_{L^2}+\|\na\dot u \|_{L^2}\right),
\ea\ee
which, together with   the Gagliardo-Nirenberg inequality  \eqref{g2} for $q=2$, yields  that for $r\triangleq 4+4/\beta$ and $\de_0\triangleq (2r+(1-\beta)(r-2))/(3r-4)\in(0,1),$
\be\la{lbqnew-jia32}\ba
b(t)&\le -\frac{1}{2\mu+\lambda} \int_0^t\n Fdt\\
&\le C(\on)\int_0^{\si(T)} \si^{(\beta-1)(r-2)/(4(r-1))}\left(\si^{1-\beta}\|F \|_{L^2}^2\right)^{(r-2)/(4(r-1))}\|\na F \|^{r/(2(r-1))}_{L^r}dt\\
&\le  C(\on,M) \int_0^{\si(T)}\si^{-(2r+(1-\beta)(r-2))/(4(r-1))}\left(\si^2\|\na F\|^{2}_{L^r} \right)^{r/(4(r-1))}dt\\
&\le  C(\on,M) \left(\int_0^{\si(T)}\si^{-\de_0}dt\right)^{(3r-4)/(4(r-1))}  \left(\int_0^{\si^(T)} \si^2\|\na F\|^{2}_{L^r}  dt\right)^{r/(4(r-1))}\\ 
&\le  C(\on,\bar m,M) \left( \int_0^{\si(T)} \left(\si^2\|\n^{1/2}\dot u\|^{2}_{L^2} +\si^2\|\na\dot u\|^{2}_{L^2}\right)dt \right)^{r/(4(r-1))}\\ 
&\le  C(\on,\bar m,M) C_0^{r/(4(r-1))} ,
\ea\ee
where in the second, fourth, and last inequalities one has used respectively \eqref{uv1}, \eqref{z.2}, and   \eqref{h27}. This combined with \eqref{z.3} yields that
   \be\la{a103}
   \sup_{t\in
[0,\si(T)]}\|\n\|_{L^\infty} \le \on
+C(\on,\bar m,M)C_0^{1/4}\le\frac{3 \bar\n  }{2},\ee
 provided $$C_0\le \ve_1\triangleq\min\{1, (\on/(2C(\on,\bar m,M)))^{4}\}. $$

Furthermore, it follows from \eqref{h19} and \eqref{ly8} that for $t\in[\si(T),T],$
\be \la{hg2} \ba
\|F\|_{H^1}
&\le C\left(\|\na u \|_{L^2}+\|P\|_{L^2}+\|m^2\|_{L^2}+\|\n\dot u\|_{L^2}\right)\\
&\le C(\on,\bar m) C_0^{1/2}t^{-1/2},
\ea\ee
which together with \eqref{g2} and \eqref{z.2} shows
\be\la{hg1}\ba
\int_{\si(T)}^T\|F\|_{L^\infty}^4dt
&\le C\int_{\si(T)}^T\|F\|_{L^{72}}^{\frac{35}{9}}\|\na F\|_{L^{72}}^{\frac{1}{9}}dt\\
&\le C(\on,\bar m,M)C_0^{35/18}\int_{\si(T)}^T  t^{-\frac{25}{18}}(\|\n^{1/2}\dot u\|_{L^2} +\|\na\dot u\|_{L^2} )^{\frac{1}{9}}dt\\
&\le C(\on,\bar m,M)C_0^{35/18},
\ea\ee
where in the last inequality, one has used \eqref{z1} and \eqref{a16}.
This shows that for all $\si(T)\le t_1\le t_2\le T,$
\bnn\ba
b(t_2)-b(t_1) &=-\frac{1}{2\mu+\lambda} \int_{t_1}^{t_2}\n \left(F+m^2\right)dt\\
&\le  C(\on) \int_{t_1}^{t_2} \|F\|_{L^\infty}dt  \\
&\le \frac{1}{2\mu+\lambda}(t_2-t_1)+ C(\on,M) \int_{\si(T)}^T \|F\|_{L^\infty}^4dt\\
&\le \frac{1}{2\mu+\lambda}(t_2-t_1)+  C(\on,\bar m,M) C_0^{35/18}   ,
\ea\enn
which implies that  one can
choose $N_1$ and $N_0$ in (\ref{a100}) as:
\bnn
N_1=\frac{ 1}{2\mu+\lambda},\quad  N_0=C(\on,\bar m,M)C_0^{35/18}.
\enn
Hence, we set $\bar\zeta= 1 $ in (\ref{a101}) since for all $  \zeta \ge 1,$
$$
g(\zeta)=-\frac{ \zeta^{\ga+1}}{2\mu+\lambda} \le -N_1=- \frac{ 1}{2\mu+\lambda}.
$$
Lemma \ref{le1} and (\ref{a103}) thus leads to
\be\la{a102}
\sup_{t\in[\si(T),T]}\|\n\|_{L^\infty}\le  \frac{ 3\bar \n}{2}  +N_0 \le\frac{7\bar \n }{4},
\ee
provided
$$
C_0\le\min\{\ve_1,\ve_2 \}, \quad\mbox{ for}\ve_2\triangleq \left(\frac{ \bar \n }{4C(\on,\bar m,M) }\right)^{18/35}.
$$
The combination of (\ref{a103}) with (\ref{a102}) yields
\begin{align}\label{nsj}
\sup_{0\le t\le T}\|\n(t)\|_{L^\infty}\le\frac{7\bar \n }{4}.
\end{align}

Finally, it follows from \eqref{mfc} and \eqref{hj1} that
$$
D_t|m|=-\frac{|m|^3}{2\mu+\lambda}-\frac{|m|(\rho^\gamma+F)}{2\mu+\lambda},
$$
which together with Lemma \ref{le1}, we also can obtain
$$
\sup_{0\le t\le T}\|m(t)\|_{L^\infty}\le\frac{7\bar m }{4},
$$
provided that
$$ C_0\le\min\{\ve_3,\ve_4 \}, \quad\mbox{ for}\
\ve_3\triangleq\left(\frac{\bar m}{2C(\on,\bar m,M)}\right)^{4},\
\ve_4\triangleq \left(\frac{ \bar m }{4C(\on,\bar m,M) }\right)^{18/35},$$
which combining with \eqref{nsj} yields \eqref{lv102} and completes the proof of Lemma \ref{le7}.
\end{proof}

With Lemma \ref{le5} and Lemma \ref{le7} at hand, we are now in a position to prove Proposition \ref{pr1}.

\textbf{Proof of Proposition  \ref{pr1}.}
It follows from \eqref{h27} that
\be\la{end-h27}
A_1(T)+A_2(T)\le  C_0^{1/2},
\ee
provided
$$
C_0\le\varepsilon_5\triangleq(C(\bar{\n},\bar m))^{-2}.
$$


Letting $\varepsilon\triangleq\min\{\varepsilon_0, \varepsilon_5\}$, we obtain \eqref{z2} directly from \eqref{lv102}, \eqref{end-h27} and  finish the proof of Proposition  \ref{pr1}.

\section{A priori estimates (II): higher order estimates }\la{se5} 
From now on, for smooth initial data $(\n_0,u_0,m_0)$ satisfying \eqref{co1} and \eqref{h7}, assume that $(\n,u,m)$ is a smooth solution of (\ref{a1})-(\ref{h2}) on $\O \times (0,T]$ satisfying (\ref{z1}).
Then, we derive some necessary uniform  estimates on the spatial gradient of the smooth solution $(\n,u,m)$.
\begin{lemma}\la{le4}
There is a positive constant $C$ depending only on $T, \mu, \lambda, \ga, a, \on, \bar m,$ $\beta, N_0,$ $M, q,$  $\|\n_0\|_{H^1\cap W^{1,q}},$ $\|m_0\|_{H^1\cap W^{1,q}}$ and $\|\nabla u_0\|_{L^2}$ such that
\be\la{pa1}
\ba
&\sup_{0\le t\le T}\left(\norm[H^1\cap W^{1,q}]{\rho}+\|m\|_{H^1\cap W^{1,q}}+\|\nabla u\|_{L^2}+ t\|\na^2 u\|^2_{L^2}\right)\\
&+\int_0^T \left(\|\nabla^2  u\|_{L^2}^2+\|\nabla^2 u\|_{L^q}^{(q+1)/q}+t\|\nabla^2 u\|_{L^q}^2 \right)dt\le C.
\ea
\ee
\end{lemma}
\begin{proof}
First, \eqref{h29} together with Gr\"onwall's inequality, \eqref{ly12}, \eqref{h27}, \eqref{a16}, and \eqref{z1} yields
\be\ba\la{ah31}
&\sup\limits_{0\le t\le T}\left( \| \na u\|_{L^2}^2+ \|P\|_{L^2}^2+\| m^2 \|_{L^2}^2\right) \\
&+\int_{0}^{T} \int  \left(\n|\dot{u} |^2+ P^3 +m^6 \right)dxdt\le C.
\ea\ee

Next, multiplying \eqref{lv3.46} by $t$ and integrating the resulting inequality over $(0,T)$ combined with \eqref{a16} and \eqref{ah31} lead to
\be\la{b19}\ba
\sup_{0\le t\le T}t \|\n^{1/2}\dot{u}\|_{L^2}^2 +\int_0^Tt\|\na\dot{u}\|_{L^2}^2dt\le C.
\ea\ee

On the one hand, for $   p\in [2, q],$ $|\nabla\n|^p$ satisfies
\begin{equation}
\begin{split}
& (|\nabla\n|^p)_t + \text{div}(|\nabla\n|^pu)+ (p-1)|\nabla\n|^p\text{div}u  \\
 &+ p|\nabla\n|^{p-2}(\nabla\n)^t \nabla u (\nabla\n) +
p\n|\nabla\n|^{p-2}\nabla\n\cdot\nabla\text{div}u = 0.
\end{split}
\end{equation}
Thus,
\be\la{L11}\ba
\frac{d}{dt} \norm[L^p]{\nabla\n}
&\le C(1+\norm[L^{\infty}]{\nabla u} )\norm[L^p]{\nabla\n} +C\|\na^2u\|_{L^p}\\
&\le C(1+\norm[L^{\infty}]{\nabla u} )\norm[L^p]{\nabla\n} +C\|\n\dot u\|_{L^p}+C\|\nabla m^2\|_{L^p},
\ea\ee
due to
\be
\la{ua1}\|\na^2 u\|_{L^p}\le   C\left(\|\n\dot u\|_{L^p}+ \|\nabla
P \|_{L^p}+\|\na m^2\|_{L^p}\right),\ee
which follows from the standard
$L^p$-estimate for the following elliptic system:
 \bnn  \mu\Delta
u+(\mu+\lambda)\na {\rm div}u=\n \dot u+\na P+\na m^2,\quad \, u\rightarrow
0\,\,\mbox{ as } |x|\rightarrow \infty. \enn

On the other hand, by a similar calculation, it follows from \eqref{a1}$_3$ that
\be\la{aL11}\ba
\frac{d}{dt} \norm[L^p]{\nabla m}\le C(1+\norm[L^{\infty}]{\nabla u} )\norm[L^p]{\nabla m} +C\|\n\dot u\|_{L^p}+C\|\nabla P\|_{L^p}.
\ea\ee

Combining \eqref{L11} and \eqref{aL11}, we can obtain
\be\la{aaL11}\ba
\frac{d}{dt} (\norm[L^p]{\nabla \n}+\norm[L^p]{\nabla m})\le C(1+\norm[L^{\infty}]{\nabla u} )(\norm[L^p]{\nabla \n} +\norm[L^p]{\nabla m}) +C\|\n\dot u\|_{L^p}.
\ea\ee

Next, it follows from the Gagliardo-Nirenberg inequality, \eqref{ah31}, and \eqref{h19} that
\be\la{419}\ba
&\|\div u\|_{L^\infty}+\|\o\|_{L^\infty}\\
&\le C \|F\|_{L^\infty}+C\|P\|_{L^\infty}+C\|m^2\|_{L^\infty} +C\|\o\|_{L^\infty}\\
&\le C(q) +C(q) \|\na F\|_{L^q}^{q/(2(q-1))} +C(q) \|\na \o\|_{L^q}^{q/(2(q-1))}\\ &\le C(q) +C(q) \|\n\dot u\|_{L^q}^{q/(2(q-1))},
\ea\ee
which, together with Lemma \ref{le9}, \eqref{ua1} and \eqref{ah31}, yields that
\be\la{b24}\ba
\|\na u\|_{L^\infty }\le
&C\left(\|{\rm div}u\|_{L^\infty }+\|\o\|_{L^\infty } \right)\log(e+\|\na^2 u\|_{L^q})+C\|\na u\|_{L^2}+C\\
\le& C\left(1+\|\n\dot u\|_{L^q}^{q/(2(q-1))}\right)\log(e+\|\rho \dot u\|_{L^q} +\|\na \rho\|_{L^q}+\|\nabla m \|_{L^q})\\
\le& C\left(1+\|\n\dot u\|_{L^q}\right)\log(e+\|\na \rho\|_{L^q}+\|\nabla m \|_{L^q}).
\ea\ee

Next, from H\"older's inequality and \eqref{z.2}   that
\bnn\la{b22}\ba
\| \rho \dot u\|_{L^q} & \le\| \rho \dot u\|_{L^2}^{2(q-1)/(q^2-2)}\|\n\dot u\|_{L^{q^2}}^{q(q-2)/(q^2-2)}\\
& \le C \| \rho \dot u\|_{L^2}^{2(q-1)/(q^2-2)}\left(\| \rho^{1/2} \dot u\|_{L^2}+\|\na\dot u\|_{L^2}\right)^{q(q-2)/(q^2-2)}\\
& \le C \| \rho^{1/2}  \dot u\|_{L^2} +C \| \rho^{1/2} \dot u\|_{L^2}^{2(q-1)/(q^2-2)}\|\na \dot u\|_{L^2}^{q(q-2)/(q^2-2)},
\ea\enn
which combined with \eqref{ah31} and\eqref{b19} implies that
\be\la{4a2}\ba
&\int_0^T\left(\|\rho \dot u\|_{L^q}^{1+1 /q}+t\| \rho \dot u\|_{L^q}^2\right)dt\\
&\le C \int_0^T\left( \| \rho^{1/2}  \dot u\|_{L^2}^2+ t\|\na \dot u\|_{L^2}^2\right)dt
+C\int_0^{\sigma(T)}t^{-\frac{\kappa}{\kappa+2}}dt+C \\
&\le C.
\ea\ee
where $\kappa=q(q+1)(q-2)$.

Then,
 substituting \eqref{b24} into \eqref{aaL11} where $p=q$, we deduce from Gr\"onwall's inequality and \eqref{4a2} that
\be \la{b30}
\sup\limits_{0\le t\le T}(\|\nabla\rho\|_{L^q}+\|\nabla m\|_{L^q})\le C,
\ee
which, along with \eqref{ua1} and \eqref{4a2}, shows

\be\la{mhd7}
\int_0^T\left(\|\na^2u\|_{L^q}^{(q+1)/q}+t\|\na^2u\|_{L^q}^2\right)dt\leq C.
\ee

Finally, taking $p=2$ in (\ref{L11}), one gets by using \eqref{ah31}, \eqref{mhd7}, and Gr\"onwall's inequality that
\bnn
\sup\limits_{0\le t\le T}(\|\nabla \n\|_{L^2}+\|\nabla m\|_{L^2})\le C,
\enn
which, together with (\ref{b30}), (\ref{ua1}), (\ref{b19}), and (\ref{mhd7}), yields (\ref{pa1}). The proof of Lemma \ref{le4} is completed.
\end{proof}

Next, we are going to show the following spatially weighted mean estimate of $m$ and $\rho$, which was proved in Lemma \ref{lemma2.7} and \cite[Lemma 4.2]{Li201977}.
\begin{lemma}\la{le6}
There is a positive constant $C$ depending only on $T, \mu, \lambda, \ga, a, \on, \bar m, \beta, N_0, M, q,$  $\|\na(\bar x^a\n_0)\|_{L^2\cap L^q},$ $\|m_0\|_{H^1\cap W^{1,q}}$ and $\|\nabla u_0\|_{L^2}$ such that
\be\la{q}\ba
&\sup_{0\le t\le T}\big(\|\bar x^a m\|_{L^1}+\|\bar x^a\n\|_{L^1\cap H^1\cap W^{1,q}}\big)\le C.
\ea\ee
\end{lemma}

\begin{lemma}\label{lem4.5v}
There is a positive constant $C$ depending only on  $T, \mu, \lambda, \ga, a, \on, \bar m, \beta, N_0, M, q,$  $\|\na(\bar x^a\n_0)\|_{L^2\cap L^q},$ $\|m_0\|_{H^1\cap W^{1,q}}$ and $\|\nabla u_0\|_{L^2}$ such that
\begin{equation}\la{gj13}\ba
\sup_{0\leq t\leq T }(t\|\na u\|^2_{H^1}+t\|\n^{1/2}u_t\|^2_{L^2})+\int_0^Tt\|\na u_t\|_{L^2}^2dt\le C.\ea
\end{equation}
\end{lemma}
\begin{proof}
First, the combination of \eqref{ljo1} with \eqref{q} gives  that
for any $\eta\in(0,1]$ and any $s>2,$
 \be\la{5.d2}\|\n^\eta u \|_{L^{s/\eta}}+ \|u\bar x^{-\eta}\|_{L^{s/\eta}}\le C(\eta,s,T).\ee

Next, differentiating $\eqref{a1}_2$ with respect to $t$ gives
\be\la{zb1}\ba
&\n u_{tt}+\n u\cdot \na u_t-\mu\Delta u_t-( \mu+\lm)\na  \div u_t  \\
&=-\n_t(u_t+u\cdot\na u)-\n u_t\cdot\na u -\na P_t-(\nabla m^2)_t.
\ea\ee
Multiplying \eqref{zb1} by $u_t$ and integrating  over $\O,$ we obtain after using  $\eqref{a1}_1$ that\be\ba  \la{na8}&\frac{1}{2}\frac{d}{dt} \int \n |u_t|^2dx+\int \left(\mu|\na u_t|^2+( \mu+\lm)(\div u_t)^2  \right)dx\\
&=-2\int \n u \cdot \na  u_t\cdot u_tdx  -\int \n u \cdot\na (u\cdot\na u\cdot u_t)dx\\
&\quad-\int \n u_t \cdot\na u \cdot  u_tdx
+\int P_{t}\div u_{t} dx+\int(m^2)_t\div u_tdx\triangleq \sum_{i=1}^{5}K_i.
\ea\ee

Similar to the proof of  \cite[Lemma 4.3]{Li201977}, we have
\be\no
\sum_{i=1}^{4}K_i \le  \de\| \na u_{t}\|_{L^{2}}^{2}+C(\de)  \left(\| \na^{2} u \|_{L^{2}}^{2} +  \|\n^{1/2} u_{t}\|_{L^{2}}^{2}+1\right).
\ee
For the term $K_5$,   we obtain after integration by parts  that
\be\ba\no
K_5&=-\int (u\cdot\nabla m^2\div u_t+2m^2\div u\div u_t)dx\\
&\leq C\|u{\bar x}^{3(2-q)a/7q}\|_{L^{14q/(q-2)}}\|m{\bar x}^{3(q-2)a/7q}\|_{L^{7q/3(q-2)}}\|\nabla m\|_{L^q}\|\nabla u_t\|_{L^2}\\
&\quad+C\|m^2\|_{L^\infty}\|\nabla u\|_{L^2}\|\nabla u_t\|_{L^2}\\
&\leq \de\| \na u_{t}\|_{L^{2}}^{2}+C(\de)  .
\ea\ee

Then, putting $K_1$-$K_5$ into \eqref{na8} and choosing $\de$ suitably  small, we obtain after  using \eqref{ua1} and \eqref{pa1} that
\be \la{5.av2}
\frac{d}{dt} \int \n |u_t|^2dx+\mu\int   |\na u_t|^2 dx   \le   C \int \n |u_t|^2dx+C\int \n |\dot u |^2dx+1.
\ee
It follows from \eqref{ua1}  and \eqref{5.d2} that
\be\ba
&\|\na u\|_{H^1}+\|\n^{1/2} u\cdot\na u\|_{L^2}\\&\le C+C\|\n^{1/2}  \dot u\|_{L^2}+ C\|\n^{1/2} u\|_{L^6}\|\na u\|_{L^2}^{2/3}\|\na^2u\|_{L^2}^{1/3}\\
&\le C +C\|\n^{1/2}  \dot u\|_{L^2}+\frac12\|\na^2u\|_{L^2},
\ea\ee
which, together with \eqref{ah31}, shows
\be \la{5.av3}
\|\na u\|_{H^1}+\|\n^{1/2} u_t\|_{L^2}\le C\|\n^{1/2} \dot u\|_{L^2}+C.
\ee
This combined with \eqref{5.av2}, \eqref{ah31}, and Gr\"onwall's inequality gives \eqref{gj13} and finishes the proof of Lemma \ref{lem4.5v}.
\end{proof}

\section{Proof of Theorem \ref{th1}}\la{se4}

With all the a priori estimates in Sections \ref{se3} and \ref{se5} at hand, we are
ready to prove the main result of this paper.

By Lemma \ref{th0}, there exists a
$T_*>0$ such that the Cauchy problem (\ref{a1})--(\ref{h2}) has a unique strong solution $(\n,u,m)$ on $\O \times
(0,T_*]$. We will use the a priori estimates, Proposition \ref{pr1}
and Lemmas \ref{le4}-\ref{lem4.5v}, to show that the local strong solution exists for all time.

First,  it follows from (\ref{As1}), (\ref{As2}), and (\ref{co1})    that
$$ A_1(0)+A_2(0)=0,  \quad \n_0\leq\bar{\n},  \quad |m_0|\leq\bar{m}.$$   Therefore, there exists a
$T_1\in(0,T_*]$ such that (\ref{z1}) holds for $T=T_1$.

Next, set \bn \la{s1}T^*=\sup\{T\,|\,{\rm (\ref{z1}) \
holds}\}.\en Then $T^*\geq T_1>0$. Hence, for any $0<\tau<T\leq T^*$
with $T$ finite, one deduces from   \eqref{gj13}
that
 \be \la{sp43}\na u \in C([\tau ,T];L^2\cap L^q),\ee
where one has used the standard embedding
$$L^{\infty}(\tau,T; H^1)\cap H^1(\tau,T; H^{-1})\hookrightarrow C(\tau,T; L^q),\quad\texttt{for~any}~q\in[2,\infty).$$
Moreover, it follows from \eqref{pa1}, \eqref{q} and  \cite[Lemma 2.3]{L2} that
\be \la{n20}\n\in C([0,T];L^1\cap H^1\cap W^{1,q}),\quad m\in C([0,T]; H^1\cap W^{1,q}).\ee

Finally,   we claim that \be \la{s2}T^*=\infty.\ee Otherwise,
$T^*<\infty$. Then by Proposition \ref{pr1}, (\ref{z2}) holds for
$T=T^*$. It follows from  \eqref{a16}, \eqref{q},   \eqref{sp43} and
(\ref{n20})  that $(\n(x,T^*),u(x,T^*),m(x,T^*))$ satisfies
(\ref{co1})   except $u(\cdot,T^*)\in \dot H^\beta.$
  Thus, Lemma
\ref{th0} implies that there exists some $T^{**}>T^*$ such that
(\ref{z1}) holds for $T=T^{**}$, which contradicts (\ref{s1}).
Hence, (\ref{s2}) holds. Lemmas \ref{th0} and \ref{le4}-\ref{lem4.5v}
  thus show that $(\n,u,m)$ is in fact the unique
strong  solution defined on $\O \times(0,T]$ for any
$0<T<T^*=\infty$. Thus, the proof of Theorem  \ref{th1} is completed.

\section*{Conflict-of-interest statement}
All authors declare that they have no conflicts of interest.

\section*{Data Availability}
No data were used for the research described in the article.

\section*{Acknowledgments}
J. Li is supported by the National Natural Science Foundation of China (No.12601413). B. L\"u is supported by the National Natural Science Foundation of China (No.12371219), and the Double-Thousand Plan of Jiangxi Province (No. jxsq2023201115).

\end{document}